\documentclass[11pt,oneside]{amsart}
\usepackage{amscd,amsmath,amssymb,amsfonts}
\usepackage[all]{xy}
\usepackage{hyperref}
\usepackage{url}
\usepackage{stmaryrd}

\usepackage{soul}

\usepackage[a4paper, total={6in, 8in}]{geometry}

\usepackage{xcolor}

\usepackage{upgreek}

\theoremstyle{plain}
\newtheorem{thm}{Theorem}
\newtheorem{lem}[thm]{Lemma}
\newtheorem{cor}[thm]{Corollary}
\newtheorem{prop}[thm]{Proposition}

\theoremstyle{definition}
\newtheorem{ex}[thm]{Example}
\newtheorem{defn}[thm]{Definition}

\newtheorem{fact}[thm]{Fact}
\newtheorem{rmk}[thm]{Remark}
\newtheorem{rmks}[thm]{Remarks}
\numberwithin{thm}{section}
\numberwithin{equation}{section}

\newcommand{\ga}[2]{\begin{gather}\label{#1}#2 \end{gather}}

\newcommand{\sE}{{\mathsf E}}

\newcommand{\sO}{{\mathcal O}}

\newcommand{\C}{{\mathbb C}}

\newcommand{\F}{{\mathbb F}}
\newcommand{\G}{{\mathbb G}}

\renewcommand{\P}{{\mathbb P}}
\newcommand{\Q}{{\mathbb Q}}

\newcommand{\Z}{{\mathbb Z}}

\newcommand{\Pb}{{\mathbb{P}}}
\newcommand{\Cb}{{\mathbb{C}}}
\newcommand{\Zb}{{\mathbb{Z}}}
\newcommand{\Lb}{{\mathbb{L}}}
\newcommand{\Qb}{{\mathbb{Q}}}

\begin{document}

\title{The Grothendieck-Katz Conjecture for privileged local systems}
\author{H\'el\`ene Esnault \and Michael Groechenig}
\address{Freie Universit\"at Berlin, Berlin,  Germany}
\email{esnault@math.fu-berlin.de }
\address{Department of Mathematics, University of Toronto}
\email{michael.groechenig@utoronto.ca}

\maketitle

\begin{abstract}
We define  study privileged local systems  on a smooth  quasi-projective complex  variety $X$ with fixed quasi-unipotent monodromies around a boundary divisor.   The notion is a generalization of Katz' physical rigidity on an open of $\mathbb P^1$. It is defined in any dimension and 
 includes non-rigid local systems.   In the latter case,  Katz used the \emph{middle convolution} procedure to classify all physically local systems and derived several consequences such as the $p$-curvature conjecture for local systems of this type. We present a new proof, which avoids the use of middle convolution. This yields examples in higher dimension of  local systems which verify the $p$-curvature conjecture. 
\end{abstract}

\medskip

\section{Introduction}\label{sec:intro}

Let $X$ be a dense open subset of the complex projective line.  In \cite{Kat96}  Katz studies irreducible 
local systems which are rigid. He shows that they are uniquely determined by their monodromies at infinity (which he calls the ``physical rigidity''  property), he classifies them entirely. They are coming from characters via the six functor formalism and Fourier transforms (to paraphrase in modern language), and when the monodromies at infinity are quasi-unipotent, this character is torsion. In this case they are of geometric origin, that is subquotients of a Gauss-Manin system stemming from a family of smooth projective varieties. A generalization of this result for $G$-local systems was recently given in \cite{Fae24}, using methods from the geometric Langlands programme.
For Gauss-Manin systems, Katz in \cite{Kat82} proved the $p$-curvature  conjecture: if the Gauss-Manin connection  of a spreading-out has $p$-curvature equal to zero (that is if it is spanned by flat sections) for almost all $p$, then its monodromy is finite. 
His proof is purely geometric. He shows that the $p$-curvature condition implies that the Hodge filtration of the Gauss-Manin connection is stabilized by the connection, thus the connection is unitary, thus finite as the local system is defined over $\mathbb Z$.
 It is an open question if arbitrary subquotients of a Gauss-Manin connections satisfy the $p$-curvature conjecture.

 We now drop the assumption $X \subset \Pb^1_{\Cb}$.  We consider $X$ to be a smooth quasi-projective complex variety of arbitrary dimension $d$ and denote by $\bar{X}$ a projective completion, which is itself smooth, such that the boundary $\bar{X}\setminus X =D$ is a simple normal crossings divisor with components $D=\bigcup_{i \in I} D_i$. For $i \in I$, we denote by $\gamma_i \in \pi_1(X,x)$ a loop around $D_i$, which is well-defined up to conjugation. We fix an $I$-tuple of conjugacy classes $[R_i] \subset {\mathsf{GL}}_r(\Cb)$.
\begin{defn}
We denote by $M_B(X,r, [R_i])$ the moduli space of irreducible local systems of the fundamental group, defined by a representation  $\rho\colon \pi_1(X,x) \to {\mathsf{GL}}_r(\Cb)$, such that $\rho(\gamma_i) \in [R_i]$. An isolated point $[\rho] \in M_B(X,r,[R_i])$ is called a \emph{rigid local system}. One calls $[\rho]$ \emph{cohomologically rigid} (also known as \emph{infinitesimally rigid}) if it is a reduced and isolated point, which is equivalent  to the vanishing of the tangent space and thereby expressible as a cohomological condition.
\end{defn}

While rigid local systems $X$ are conjectured by Simpson to  be of \emph{geometric origin} (\cite{Sim92}), this problem remains open. However, various consequences of this conjectural property are unconditionally proven.   The rigid connections on $X$ smooth  quasi-projective  are defined over a number ring if they are cohomologically rigid (\cite{EG18}).  Only assuming rigidity, 
one can define $F$-isocrystals on mod $p$ reduction of $X$ (see \cite{EG25} and its precursors \cite{EG18,EG21}).

In \cite[Section~8.4]{EG20}, the authors made partial progress on the $p$-curvature conjecture for \emph{rigid} flat connections  for $X$ projective, by establishing that rigidity paired with the assumption of the $p$-curvature conjecture implies that the monodromy is contained within a compact subgroup. Under the additional assumption of cohomological rigidity one also obtains integrality and thus the assertion that the monodromy factors through a conjugate of ${\mathsf{GL}}_r(\mathcal{O}_K)$, where $K$ is a number field. However, this only allows one to infer the finiteness of monodromy conjectured by the $p$-curvature conjecture, if all the Galois conjugates of the local system can be shown to have unitary monodromy as well. We remark that there are examples of Gauss-Manin connections for which one irreducible summand is unitary whereas another one is not.

The aim of this note it to address this  precise issue for so-called {\it privileged local systems}, a notion which we introduce in Definition~\ref{defn:phys_rig}
in any dimension and which encompasses  Katz's \emph{physical rigidity}, thus rigidity on a rational curve.  
The decisive property of privileged local systems underlying our argument is that they are completely controlled by the conjugacy classes $[R_i]$ describing the monodromy around the components of the boundary divisors.
 This leads to the insight (see
Theorem~\ref{thm:arith_frob}) that a $p$-Frobenius lift in the Galois group of the field of definition of a privileged local system   
acts as the crystalline  Frobenius 
of the crystal defined by the underlying connection mod $p$. To speak figuratively, our key observation is that the  {\it  the Frobenius on the Betti coefficients equals the crystalline Frobenius.}

 Equipped with this observation we can prove the main Theorem~\ref{thm:main}  of our article: {\it privileged local systems verify the $p$-curvature conjecture.} For the proof, we have to eliminate the poles of the connection so as to be able to apply Simpson's theory culminating in the identification of unitary local systems on $X$ projective with $\mathbb C^\times$-fixed Higgs bundles with underlying stable vector bundles, also known as the Narasimhan--Seshadri correspondence. While the idea developed in \cite[Section~8.4]{EG20} to prove unitarity was to find good $p$'s which yield an $F$-crystal, here  rather the Frobenius is exchanging the finitely many crystals associated to the privileged local systems in a given rank, much in the spirit of our counting argument in the proof of the integrality in \cite{EG18} and in the proof of the existence of $p$-to-$p$ companions for rigid local systems when $X$ is projective in \cite[Section~7.3]{EG20}. Indeed, we also prove in Proposition~\ref{prop:comp} the existence of $p$-to-$p$ companions for privileged local systems under the simplifying assumption that the matrices $R_i$ are finite.

 Our main theorem has various applications. Isomonodromic deformations of privileged local systems exist and verify the $p$-curvature conjecture, see Theorem~\ref{thm:gk}. Moreover,  subquotients of connections yielding  $F$-crystals  when the local system is defined over $\mathbb Z$ also verify the $p$-curvature conjecture, see Proposition~\ref{prop:V}. This applies   to summands of Gauss-Manin connections when they are privileged with respect to the other summands, see Theorem~\ref{thm:subquotient}   \\  \\
{\it Acknowledgements:} It has been a great privilege to devote ourselves to the fascinating circle of ideas of rigid local systems. Our intellectual debt to work by Katz and Simpson will be evident to all readers. Furthermore, we are grateful to  everyone with whom we had the pleasure to  engage in discussions 
on topics addressed in this note, both past and present.  This includes Yohan Brunebarbe, François Charles, Dustin Clausen, Joakim F\ae rgeman, Daniel Litt, Alexander Petrov, Peter Scholze, Ananth Shankar, Emanuel Ullmo. The first author thanks the IH\'ES for its warm hospitality during the preparation of this work.

\section{Privileged local systems} \label{sec:phy_rig}
Let $X$ be a smooth quasi-projective variety of dimension $d$ defined over the field of complex numbers. Let $j \colon X\hookrightarrow \bar X$ be a good compactification, i.e., $\bar X$ is a smooth projective variety and $\bar X\setminus X=: D=\cup_{i\in I} D_i$ is a strict normal crossings divisor. Let $\mathbb L$ be an {\it irreducible} local system on $X$ of rank $r$. Let  $[R_i]\subset {\rm GL}_r(\C)$ be the conjugacy classes of its monodromies along the $D_i, \ i\in I$ and $\delta$ be its determinant character. 

\subsection{Classical notions of rigidity} \label{ss:rig}
When   $\bar{X} = \Pb^1_{\Cb}$,  Katz  in \cite[1.1]{Kat96} defined $\mathbb L$ to be  {\it physically rigid}  if any other  irreducible local system with the same  conjugacy classes of its monodromies along the $D_i, \ i\in I$ is isomorphic to $\mathbb L$.  He proved \cite[Theorem~1.1.2]{Kat96} that this condition is cohomological and equivalent to the coarse Betti moduli space $M_B(X,r, [R_i])$ of irreducible complex local systems of rank $r$ with $[R_i]$ monodromies along the $D_i$
being reduced to a smooth point. Said differently, an irreducible local system $\mathbb L$ on $\Pb^1_{\Cb} \setminus D$ is  physically rigid 
 if and only if it is {\it cohomologically rigid} (i.e. is a smooth isolated point of $M_B(X,r,[R_i])$) if and only if it is  {\it rigid} (i.e., is an isolated point of $M_B(\Pb^1_{\Cb} \setminus D,r,[R_i])$).  We remark that the topological fundamental group  of $\Pb^1_{\Cb} \setminus D$ is spanned by the  conjugacy classes of small loops  around the $D_i$.  In  this discussion, we could also fix a torsion character $\delta \colon \pi_1(X,x) \to \mu_n\subset \C^\times$ and replace $M_B(X,r, [R_i])$ by $M_B(X,r,[R_i], \delta)$, the coarse moduli space of irreducible local systems of rank $r$ with $[R_i]$ monodromies along the $D_i$ and determinant $\delta$.

Katz's main theorem  \cite[9.3.3]{Kat96} states that rigid local systems with fixed quasi-unipotent monodromies $[R_i]$ along the $D_i$ are of geometric origin; more precisely that they arise through iteratively applying his \emph{middle convolution} procedure to a torsion character. We remark that by Deligne's semi-simplicity Theorem \cite[Section~4.2]{Del71}, their Tannaka group is semi-simple.

The definitions of cohomologically rigid and  of rigid local systems extend readily to higher dimension in a straightforward way, see  \cite[Section~4]{Sim92} and \cite[Section~2]{EG18}. Recall that  by convention all  notions of rigidity require the underlying local system $\mathbb L$ to be {\it irreducible}. 
\begin{lem} \label{lem:det_rigid}
Let $X,D$ be as before (the case $D=\varnothing$ being permitted). If there is an irreducible  local system $\mathbb L$ which is rigid,  then the first Betti number $b_1(\bar X)$ of $\bar X$ is zero. If in addition the monodromies $[R_i]$ along the $D_i$ 
are quasi-unipotent, then ${\rm det}(\mathbb L)$ has finite order.
\end{lem}
\begin{proof}
If $b_1(\bar X) \neq  0$ there is a one parameter family  $\xi_t \colon  \pi_1(X,x)\to  \mathbb C^\times$ of non-torsion characters yielding a continuous non-trivial deformation  so $\mathbb L \otimes \xi_t$  of $\mathbb L$ preserving the $[R_i]$. Indeed, if  $\mathbb L\otimes \xi$ is isomorphic to $\mathbb L$, then  ${\rm det} (\xi)^r=1$ and there is no continuous deformation in $H^1(\bar X, \mu_r)$.
So  replacing ${\rm det}(\mathbb L)$ by a non-zero power, it becomes a character of the abelian quotient of ${\rm Ker}\big( \pi_1(X,x)^{\rm ab} \to \pi_1(\bar X,x)^{\rm ab} \big)$ which is spanned by the  loops around $D_i$ based at $x$. It sends those  to the determinant of  the $[R_i]$, which by our assumption has finite order. \end{proof}

 \subsection{Definition of privileged local systems} \label{defn:priv}
 
 We fix quasi-unipotent conjugacy classes $[R_i]\subset \mathsf{GL}_r(\mathbb C)$ for $i\in I$ and 
 a torsion character $\delta \colon \pi_1(X,x) \to \mu_n\subset \G_m$. We  consider the moduli space $M_B(X,r, [R_i])$,  of irreducible local systems with the conditions with monodromies $[R_i]$ along $D_i$. Recall from \cite[Proposition~2.3]{EG18}  that the Zariski tangent space at $\mathbb L\in M_B(X,r,[R_i])$ is computed by  
 \ga{}{ T_{\mathbb L} M_B(X,r, [R_i]) =H^1(\bar X, j_{!*} \sE nd(\mathbb L))= H^1(U, a_*\mathcal E nd(\mathbb L)) \notag}
 where $j\colon X\xrightarrow{a} U\xrightarrow{b} \bar X$ and $U$ is the complement of the singular locus of $D$ in 
 $\bar X$.  For example, if $\mathbb L$ extends to $\bar X$ as a local system $\mathbb L_{\bar X}$ then 
 $T_{\mathbb L} M_B(X,r, [R_i]) =H^1(\bar X, \mathcal E nd(\mathbb L_{\bar X}))$. Similarly 
 \ga{}{ T_{\mathbb L} M_B(X,r, [R_i], \delta) =H^1(\bar X, j_{!*} \sE nd^0(\mathbb L))=
 H^1(U, a_*\mathcal E nd^0(\mathbb L)) \notag} for $\mathbb L\in M_B(X,r, [R_i],\delta)$, the moduli spaces of irreducible local systems with monodromies $[R_i]$ along $D_i$ and determinant $\delta$, and 
  where $^0$ indicates the trace zero endomorphisms.

\begin{defn} \label{defn:phys_rig}

Let $(X, D, j, R_i, \delta)$ as before where we allow $D$ to be empty. 
An  irreducible local system  $\mathbb L$ on $X$  is called {\it privileged with respect to} $[R_i]$, respectively $([R_i], \delta)$  if for any $\mathbb L'\in M_B(X,r, [R_i])$, respectively  $\in  M_B(X,r, [R_i], \delta)$, with $\mathbb L'\neq \mathbb L$,  ${\rm dim}_{\mathbb C} T_{\mathbb L} M_B(X,r, [R_i])\neq {\rm dim}_{\mathbb C} T_{\mathbb L'} M_B(X,r, [R_i])$ respectively 
${\rm dim}_{\mathbb C} T_{\mathbb L} M_B(X,r, [R_i], \delta)\neq {\rm dim}_{\mathbb C}T_{\mathbb L'} M_B(X,r, [R_i],\delta)$.
\\ \ \\
We then say that $\mathbb L$ is the {\it privileged point} of $M_B(X,r, [R_i])$  respectively $M_B(X,r,[R_i], \delta)$.
\end{defn}

\begin{rmks}
\begin{itemize}
\item[1)]  If the moduli space consists of a single point, the corresponding local system is privileged. Therefore, this notion generalizes Katz's notion of physical rigidity. In particular, rigid local systems on open subsets of $\Pb^1_{\Cb}$ studied by Katz are examples of privileged local systems.
\item[2)]
The difference between the non-respective and the respective cases  lies solely on the fact that 
 for a non-trivial character $\xi$  on $X$, it could be that $\mathbb L\otimes \xi$ has the same $[R_i]$ monodromies along $D_i$, so a priori fixing the determinant  yields more examples, as we see in Example~\ref{ex}. 
 \item[3)] Privileged local systems do not have to be rigid. For example 
 if the Betti moduli space of irreducible local systems  is irreducible and has one  isolated singularity,  then this point is privileged. 
 \end{itemize}
\end{rmks}

\begin{ex} \label{ex}
Let $X$ be a  smooth projective  curve (of arbitrary genus) and $\delta$ be 
  a torsion rank one local system. Then, $\Lb = \delta$ is a privileged rank $1$ local system with respect to  $(\varnothing,\delta)$, where $\varnothing$ stands for the empty collection of conjugacy classes.
So  the existence of a privileged local system with respect to $([R_i],\delta)$ (i.e., with fixed torsion determinant $\delta$) does not allow one to conclude triviality of $H_1(\bar{X},\Qb)$.
\end{ex}
\subsection{Extending the local systems} \label{ss:good}   By Kawamata's Lemma ~\cite[Theorem~17]{Kaw81},
     there is 
   a  Cartesian  diagram
\ga{}{\xymatrix{ \ar[d]_qY \ar[r]^h & \ar[d]^{\bar q} \bar Y\\
X \ar[r]^j & \bar X
}\notag}
such that $\bar Y$ is smooth projective, $\bar Y\setminus Y=\bar q^{-1}(D)$ is a strict normal crossings divisor, $\bar q$ ramifies along a strict normal crossings divisor  $\bar q^{-1} D\cup R$,   $q$ and $\bar q$ are Galois (\cite[Lemma~5.2.4]{Mat02}) under a finite group $G$. If at  a closed point  $x$ of $D$, the local formal coordinates in $\bar X$ are $x_i, \ i=1,\ldots, d', \ d' \le d$ with $D_i$ defined by $x_i=0$, at the closed point $y$ of $\bar Y$ above $x$, there are local formal coordinates $y_i, \ i=1, \ldots d$ with $E_i:=\bar q^{-1}(D_i)$ defined by $y_i=0$ and $y_i^{n_i}=x_i$. 

For  an irreducible local system   $\mathbb L$  with {\it  finite} monodromies $[R_i]$ along the $D_i$ such that $[R_i^{n_i}]=1$ and with  finite order determinant $\delta$, there are
 local systems $\mathbb M, \ \delta_\circ $ on $\bar Y$ such that 
\ga{}{ q^*\mathbb L= h^*\mathbb M, \ q^*\delta=h^*\delta_\circ.\notag}
Hence,  $\mathbb M$  and $\delta_\circ$ are  uniquely defined on $\bar Y$,  $\mathbb M$ is semi-simple by Clifford theory, and  integral if $\mathbb L$ is. 
Furthermore $\mathbb L= q_*(\mathbb M)^G$.

\begin{lem}  \label{lem:coh}
With the notation above, and $[R_i]$ of finite order,  we have $$T_{\mathbb L} M_B(X,r, [R_i])= H^1(\bar Y, \mathcal E nd(\mathbb M))^G \  {\rm for} \  \mathbb L\in M_B(X,r,[R_i]).$$ Similarly,  there is an identification
$$T_{\mathbb L} M_B(X,r, [R_i], \delta)=  H^1(\bar Y, \mathcal E nd^0(\mathbb M))^G \ 
{\rm for}   \ \mathbb L\in M_B(X,r,[R_i], \delta).$$
\end{lem}
\begin{proof}
We restrict $q$ to $q_\circ \colon Y_\circ =q^{-1}(X_\circ)\to X_\circ$ where $Y_\circ=Y\setminus R$ and $X_\circ =X\setminus \bar q(R)$.  Set $h_\circ: Y_\circ  \xrightarrow{\alpha} V \xrightarrow{\beta} \bar Y$ where $V$ is the complement of the singular locus of $\bar q^{-1}D\cup R$ in $\bar Y$.  Then $\bar q(V)\xrightarrow{\alpha'}  U$  is an open embedding and the complement of $\bar q(V)$ in $\bar X$ is also of codimension $\ge 2$. 
As $\mathbb L$ is defined on $R\setminus D$, the equation   $H^1(U, a_* \mathcal E nd(\mathbb L))= H^1(\bar q(V),  (a \circ \alpha')_* \mathcal E nd(\mathbb L))$ holds.  On the other hand  by projection formula $H^1(\bar q(V),  (a \circ \alpha')_* \mathcal (\mathbb L^\vee \otimes \mathbb L\otimes  q_{\circ*} \mathbb C))=H^1(V,   \mathbb M^\vee \otimes \mathbb M)$. The latter  is equal to $H^1(\bar Y,   \mathbb M^\vee \otimes \mathbb M)$ as $\mathbb M$ extends to $\bar Y$. 
Thus $H^1(V,   \mathbb M^\vee \otimes \mathbb M)^G= H^1(U,   a_*(\mathbb L^\vee \otimes \mathbb L))$. This proves the first formula. As for the second one, we write 
$H^1(\bar Y, \mathcal E nd(\mathbb M))= H^1(\bar Y, \mathcal E nd^0(\mathbb M)) \oplus H^1(\bar Y, \mathbb C)$ so 
$H^1(\bar Y, \mathcal E nd(\mathbb M))^G= H^1(\bar Y, \mathcal E nd^0(\mathbb M))^G \oplus H^1(\bar Y, \mathbb C)^G$ and the latter group is equal to $H^1(\bar X, \mathbb C)=H^1(U, a_* \mathbb C)$. This finishes the proof. 
\end{proof}

\begin{rmk} \label{rmk:rootst}
In \cite[Section~3.4]{EG25} we used a root stack in place of the cover considered here. This has the advantage that it is purely canonical, here is no $G$-action to be considered and the computations with Deligne's extensions require less care. The main reason why we consider the projective $\bar Y$  cover in this note is that we use in Section~\ref{sec:proof}   Simpson's theorem according to which on $\bar Y$, a 
polarized complex  variation of Hodge structure with underlying polystable vector bundle with vanishing numerical Chern classes is a direct sum of unitary connections. This is well documented when $\bar Y$ is projective,  see \cite[Section~4]{Sim92}.

\end{rmk}
\subsection{Privileged local systems and integrality} \label{sec:int}
As by Corollary~\ref{cor:cyclo} privileged local systems are in particular defined over a number field, 
the proof of \cite[Theorem~1.1]{EG18}, relying on Lemma~3.4 in {\it loc.cit.}, applies readily to conclude the following theorem. The proof is omitted.

\begin{thm}  \label{thm:int} 
A privileged local system $\mathbb L$  with respect  $[R_i]$, respectively  $([R_i], \delta)$ with $[R_i]$ quasi-unipotent  
 along the $D_i$ and $\delta$ a torsion character  has a monodromy representation defined over the ring of algebraic integers. In particular, if $[R_i]$ have finite order, then $\mathbb M$ in Section~\ref{ss:good}  is integral as well. 
\end{thm}

\subsection{Privileged local systems  and polarized complex variations of Hodge structure} 

\begin{thm} \label{thm:tpriv}
A privileged local system $\mathbb L$  with respect  $[R_i]$, respectively  $([R_i], \delta)$ with $[R_i]$ of finite order 
 along the $D_i$ and $\delta$ a torsion character  underlies a polarized complex variation of Hodge structure. In particular $\mathbb M$ in Section~\ref{ss:good} underlies a polarized complex variation of Hodge structure as well. 

\end{thm}
\begin{proof}
On the  moduli space $M_{Dol}(\bar Y, r)$ of semi-simple Higgs bundles, respectively $M_{Dol}(\bar Y, r, \delta_\circ^H),$ where 
$\delta_\circ^H$ is the line bundle associated to $\delta_\circ$ endowed with the zero Higgs field, the $\mathbb G_m$-action is an isomorphism. Thus  $t\in \mathbb C^\times$  induces an algebraic linear  isomorphism between the cohomology groups 
$H^1(\bar Y, \mathcal E nd(\mathcal V))$ and $H^1(\bar Y, \mathcal E nd(t\cdot \mathcal V))$, and similarly between $H^1(\bar Y, \mathcal E nd^0(\mathcal V))$ and $H^1(\bar Y, \mathcal E nd^0(t\cdot \mathcal V))$.  Assume that $\mathcal V$ is the Higgs bundle associated to $\mathbb M$. We denote by $t\cdot \mathbb M$ the local system associated to $t \cdot \mathcal V$.  Thus 
$t\in \mathbb C^\times$  induces an  $\mathbb{R}$-linear  isomorphism between the cohomology groups 
$H^1(\bar Y, \mathcal E nd(\mathbb M))$ and $H^1(\bar Y, \mathcal E nd(t\cdot \mathbb M))$, and similarly between $H^1(\bar Y, \mathcal E nd^0(\mathbb  M))$ and $H^1(\bar Y, \mathcal E nd^0(t\cdot \mathbb M))$.  
On the other hand, the action of $G$ is defined on $\bar Y$, thus commutes to the action of $\mathbb G_m$ on  $\mathcal V$ and $\mathbb M$. The action of $G$ on $\mathbb M$ induces one on $t \cdot \mathbb M$, and since $\mathbb M$ is invariant under $G$, so is $t\cdot \mathbb M$ with respect to this action. 
We conclude that ${\rm dim}_{\mathbb C} H^1(\bar Y, \mathcal E nd(\mathbb M))^G= 
{\rm dim}_{\mathbb C} H^1(\bar Y, \mathcal E nd( t \cdot \mathbb  M))^G$ and likewise for the trace free parts.    Let us write $ q_{\circ *} (t\cdot \mathbb M)^G = t\cdot \mathbb L$ on $X_\circ$. Then $t\cdot \mathbb  L$ is a real analytic deformation of $\mathbb L|_{X_\circ}$. The local monodromies of $t\cdot \mathbb L$ must then have finite order dividing $n$ along $D\cup \bar q(R)$ and at the same type be real analytic deformations of the ones of  $\mathbb L$, that is $[R_i]$ along $D_i$ and $1$ along the other components meeting $X$. This implies that $t\cdot \mathbb L$ has the same monodromies at infinity as $\mathbb L$.  
Lemma~\ref{lem:coh} implies that the tangent spaces at the Betti moduli spaces for $\mathbb L$ and $t\cdot \mathbb L$ are the same. As $\mathbb L$ is privileged, we conclude $\mathbb L=t\cdot \mathbb L$. Thus $\mathbb M=q^{-1} \mathbb L=q^{-1} t\cdot \mathbb L =t\cdot \mathbb M$ on $Y$.  Thus $\mathbb M=t\cdot \mathbb M$ on $\bar Y$ . 
By Simpson's Theorem~\cite[Corollary~4.2]{Sim92} this implies that $\mathbb M$ underlies a polarized complex variation of Hodge structure, so so does $\mathbb L$. This finishes the proof.
\end{proof}

\subsection{Field of definition of privileged local systems}
Recall that if $\mathbb L$ is any local system defined by $\rho \colon \pi_1(X,x)\to {\mathsf {GL}}_r(\C)$, where $x$ is a complex point of $X$, and $\sigma \in {\mathsf{ Aut}}(\C)$, one defines 
$\mathbb L^\sigma$   by   the composition
\ga{}{ \rho^\sigma  \colon \pi_1(X,x)\xrightarrow{\rho}  {\mathsf{ GL}}_r(\C) \xrightarrow{\sigma} {\mathsf{ GL}}_r(\C). \notag} So, if 
$\mathbb L\in M_B(X,r,[R_i])(\C)$,  respectively
$\mathbb L\in M_B(X,r,[R_i], \delta)(\C)$, then 
$\mathbb L^\sigma \in M_B(X,r,[ \sigma(R_i)])(\C)$, respectively
$\mathbb L^\sigma \in M_B(X,r,[ \sigma(R_i)], \delta^\sigma)(\C)$.
Observe that $[R_i]$ is quasi-unipotent if and only if $[\sigma(R_i)]$ is. 

\begin{prop} \label{prop:sigma}
Let $(X,D, j, [R_i], \delta)$ be as in Definition~\ref{defn:phys_rig}, $\mathbb L$ be an irreducible   local system,  let $\sigma \in {\mathsf{ Aut}}(\mathbb C)$. If $\mathbb L$ is privileged
with respect to  $[R_i]$, resp.  $([R_i], \delta)$, so is 
$\mathbb L^\sigma$ with respect to  $[\sigma(R_i)]$ resp. $([\sigma(R_i)], \delta^\sigma)$.
\end{prop} 
\begin{proof}
For any irreducible local system $\mathbb L'$ in $M_B(X,r, [R_i])$, respectively $M_B(X,r, [R_i], \delta)$, $\sigma$ induces an isomorphism between 
 $ H^1(\bar X, j_{!*} \mathcal E nd(\mathbb L'))$ and $H^1(\bar X, j_{!*} \mathcal E nd(\mathbb L^{' \sigma}))$
 and likewise with the trace free endomorphisms.
\end{proof}
If  the $[R_i]$ are quasi-unipotent, they are defined over a finite cyclotomic field $K_\circ\subset \C$. By enlarging the cyclotomic field, we may  assume that a given torsion character 
$\delta \colon \pi_1(\bar{X}) \to K_{\circ}^{\times}$
is defined over $K_\circ$ as well. 

\begin{cor} \label{cor:cyclo}
Let $(X, D,j, [R_i], \delta)$ be as in Definition~\ref{defn:phys_rig}, $\mathbb L$ be a privileged local system with respect  $[R_i]$ resp. $([R_i], \delta)$ where the $[R_i]$ are quasi-unipotent and $\delta$ is a torsion character.  Then, $\mathbb L$ is defined over a finite cyclotomic field $K\subset \C$ containing $K_\circ$. 
\end{cor}
\begin{proof}
The moduli space $M_B(X,r, [R_i])$ is the categorial quotient  by ${\rm PGL}_r$  of the fine moduli space
$ {\mathsf{ Hom}}(\pi_1(X,x), {\mathsf{ GL}}_r, [R_i]) \subset  {\mathsf{ Hom}}(\pi_1(X,x), {\mathsf{ GL}}_r) $
which is defined to be the inverse image of $\prod_{i\in I}[R_i] \in  \prod_{i \in I} {\rm GL}_r(\C)$ by the restriction of $\rho \in {\mathsf{ Hom}}(\pi_1(X,x), {\mathsf{ GL}}_r)$ to fixed closed paths $\tau_i\gamma_i\tau_i^{-1}$ where the $\gamma_{i}$ are small loops turning once around $D_i$ based at a point $x_i$ near $D_i$ and $\tau_i$ is a path from $x_i$ to $x$. So, the categorial quotient   is defined over $K_\circ$ and $M_B(X,r, [R_i])$ is defined over $K_\circ$. The same holds for 
$M_B(X,r, [R_i], \delta)$. 

As $\mathbb L$ is irreducible, the fibre 
${\mathsf{Hom}} (\pi_1(X,x), {\mathsf{ GL}}_r, [R_i])_\mathbb L \subset  {\mathsf{Hom}}(\pi_1(X,x), {\mathsf{ GL}}_r)_{K_\circ}$ of the categorial quotient is a ${\mathsf{ PGL}}(K_\circ)$-torsor. Its Brauer class  splits after pulling it back  to  a possibly  larger cyclotomic extension $K\supset K_\circ$. Thus ${\mathsf{ Hom}} (\pi_1(X,x), {\mathsf{ GL}}_r, [R_i])_\mathbb L(K)\neq \varnothing$ which is equivalent to saying that $\mathbb L$ is defined over $K$. One argues similarly in the case with fixed torsion determinant $\delta$, using the moduli space $M_B(X,r, [R_i], \delta)$. 
\end{proof}

\section{$p$-curvature conjecture}
Recall that the $p$-curvature conjecture predicts that if  an integrable  connection on $X$ as in Section~\ref{sec:phy_rig} has $p$-curvature equal to $0$ in the mod $p$ reduction of $X$ for almost all $p$, then its local system of flat sections has finite monodromy.  In \cite[Theorem~9.4.1]{Kat96} Katz proves that rigid local systems with  fixed quasi-unipotent monodromies along the $D_i$ on $X\hookrightarrow \P^1_{\Cb}$ verify the $p$-curvature conjecture. He deduces it from his classification of those local systems, involving a remarkable presentation in terms of the middle convolution functor applied to torsion characters, see  Theorem 5.2.1 in {\it loc. cit.} 

The aim of this note is to prove the following  vast generalization to higher-dimensional $X $  in the absence of a general classification and possibly of  rigidity. 
\begin{thm}[Main Theorem] \label{thm:main} Let $(X,D,j, [R_i], \delta)$ be as in Definition~\ref{defn:phys_rig} with  any $d\ge 1$. The privileged local systems with respect to $[R_i]$, respectively $([R_i],\delta)$ verify the $p$-curvature conjecture where the $[R_i]$ are quasi-unipotent along the $D_i$ and $\delta$ is a torsion character. 
\end{thm}
The theorem allows  the case of empty boundary $D=\varnothing$,  amounting to a projective smooth $\Cb$-scheme $X$. In this case  the set of conjugacy classes $[R_i]$ is empty. 

 Since we do not have a classification of all privileged local systems in higher dimensions, the proof of Theorem \ref{thm:main}, which will be given in Section~\ref{sec:proof}, forcibly has to differ from Katz's proof.  
\\ \ \\
Recall that by \cite[Theorem~8.1 (iii)]{Kat82} the following general statement  holds under a weaker assumption than the one of the Grothendieck-Katz Conjecture.
\begin{prop} \label{prop:finite_inf}
If an irreducible local system $\mathbb L$ has quasi-unipotent monodromies $[R_i]$ and has $p$-curvature equal to zero for a dense set of primes, then the $[R_i]$ have finite order.
\end{prop}

\section{ Galois action of the monodromy field and crystalline Frobenius } \label{sec:galois}
\subsection{Galois action on the coefficients} \label{ss:Gal}
Let $\mathbb L$ be an irreducible local system with finite order $[R_i]$  along the $D_i$ and finite order determinant $\delta$. We denote by $K$ a cyclotomic field over which $\mathbb L$ is defined (Proposition~\ref{prop:sigma}).
For any unramified place $\frak{p}$ of $K$, the homomorphism
\ga{}{  \iota_\frak{p} \colon  \hat \Z\cdot \varphi_\frak{p} ={\mathsf{ Gal}}(K_\frak{p}/\Q_p)\to {\mathsf{ Gal}}(K/\Q) \notag}
 is defined up to conjugacy in ${\mathsf{ Gal}}(K/\Q)$, where $\varphi_\frak{p}$ is the arithmetic Frobenius of the residue field   $\F(\frak{p})$ 
 of $K_\frak{p}$. As $K$ is cyclotomic, ${\mathsf{ Aut}}(K/\Q)$ is  abelian,  thus  $\iota_\frak{p}$ is well-defined  and only depends on the prime $p$. In particular we can consider 
 $ \sigma_\frak{p}:= \iota_\frak{p}(\varphi_\frak{p}) \in {\mathsf{ Gal}}(K/\Q).$
 By definition, $ \varphi^{-1}_p$  acts on  $\sO_{K_\frak{p}}^\times$ as
 \ga{}{\varphi^{-1}_p \colon   \sO_{K_\frak{p}}^\times \to \sO_{K_\frak{p}}^\times, \ \lambda \mapsto  \lambda^{[1/p]}:=\tau\big( (  \lambda \ {\rm mod} \ \frak{p})^{1/p}\big), \notag}
 where $\tau \colon \F(\frak{p})^\times \to \mathcal O_{K_\frak{p}}^\times$ is the Teichm\"uller lift. 
 \\ \ \\
 In the following we write $K=\Q(\mu_n)$.   In particular, the orders of  the eigenvalues of the $[R_i]$ and the order of $\delta$ divide $n$. 
\begin{lem} \label{lem:arith=geom}
We choose the prime $p$ such that    $2$ and $n$ are not  divisible by $p$. Then
 $\sigma_\frak{p}^{-1}( R_i)=R_i^{1/p} \in {\mathsf{ GL}}_r(\mu_n(K))\subset {\mathsf{ GL}}_r(K) \subset {\mathsf{ GL}}_r(K_\frak{p})$  and $  \sigma_\frak{p}^{-1} \delta= \delta^{1/p} \in \mu_{n} (K)\subset K^\times \subset 
K_\frak{p}^\times.$

\end{lem}
\begin{proof} \label{lem:sigmap}

 Let $\lambda \in \mu_n(K) \subset \mu_n(K_\frak{p})$. Then 
 $ \lambda^{[1/p]}=\lambda^{1/p}\in \mu_n(K)\subset \mu_n(K_\frak{p}). $ 
 
 \end{proof}
\begin{cor}
If  $\mathbb L$ is privileged with respect to $[R_i]$, resp.  $([R_i], \delta)$, then $ \mathbb L^{\sigma_{\frak{p}}^{-1}}$ is privileged with respect to  $[R_i^{1/p}]$, resp. $([R_i^{1/p}], \delta^{1/p})$. 

\end{cor}
\begin{proof}
We combine Proposition~\ref{prop:finite_inf}, Proposition~\ref{prop:sigma} and  Lemma~\ref{lem:arith=geom}.
\end{proof}

  Given $\mathbb L$ privileged with respect to $[R_i]$  resp. $([R_i], \delta)$ with $[R_i]$ of finite order, we consider the Deligne extension 
  $\mathcal E$ on $\bar X$  of its connection on $X$ defined by the Riemann-Hilbert correspondence  \cite[Remarque~5.5]{Del70}. We simplify the notation here setting $\mathbb L'=\mathbb L^{\sigma_p^{-1}}$.
  Similarly we consider for  $\mathcal E'$
 and $\mathbb L^{\sigma_p^{-1}},$  $\mathcal N$ and $\mathcal N'$ for the ones associated to $\mathbb M$ and $\mathbb M^{\sigma_\frak{p}^{-1}}$. In particular,  by the compatibility of the Deligne extension with push-forward we have
 \ga{}{  (\bar q_* \mathcal N)^G= \mathcal E, \ (\bar q_* \mathcal N')^G= \mathcal E'.  \notag}
  As $\mathcal N$ is $G$-invariant,  by the projection formula on a dense  open of $X$ on which $q$ is étale, and from there taking the Deligne extension denoted by $\mathcal D$,  we have
  \ga{}{ \bar q_* (\mathcal N^\vee\otimes \mathcal N)^G= \mathcal D \mathcal E nd(\mathcal E), \ 
  \bar q_* (\mathcal N^{'\vee}\otimes \mathcal N')^G=\mathcal D \mathcal E nd(\mathcal E').
   \notag}
 
  \begin{fact} \label{fact:dR}  The equations
  \ga{}{ H^1_{dR}(\bar Y,   \mathcal N^\vee\otimes \mathcal N)^G
  =H^1(\bar X, j_{!*} \mathbb L^\vee\otimes \mathbb L)=  
  H^1(\bar X, j_{!*} \mathbb L^{'\vee}\otimes \mathbb L')
=
   H^1_{dR}(\bar Y,   \mathcal N^{'\vee}\otimes \mathcal N' )^G    \notag
  \notag}
hold, and similarly in the respective case taking the trace free part of the endomorphisms.
  \end{fact}
  \begin{proof}
  We apply Lemma~\ref{lem:coh} and the Riemann-Hilbert correspondence on $\bar Y$ to obtain the equality of the first two and the last two terms. The two middle terms are identified with $\sigma_{\frak{p}}^{-1}$.

  \end{proof}

 \subsection{De Rham moduli} \label{ss:dRst}
 We fix the $[R_i]$ of finite order, and some finite order character $\delta$. 
 For  each $i$, we give ourselves a set $\Gamma_i$  of $r$ rational numbers $\lambda_{ij}, \ j=1,\ldots, r $ such that no two differ by a non-zero integer and such that  $\{{\rm exp} 2\pi \sqrt{-1} \lambda_{ij}, \ j=1,\ldots, r\}$ is the set of eigenvalues of $[R_i]$.   The specific Deligne choice $0\le  \lambda_{ij}<1$ is denoted by $\Gamma_i^{\mathcal D}$. 
    We fix a rank one connection $\mathcal L$ with log poles along $D$ and residue along $D_i$ equal to ${\rm Tr}(\Gamma_i)$ such that $\mathcal L|_X^\nabla=\delta$.  For $\Gamma_i^{\mathcal D}$ we denote it by $\mathcal L^{\mathcal D}$.  If we twist $[R_i]$ to $[R^\sigma_i]$ by $\sigma\in \mathsf{Gal}(K/\mathbb Q)$, we write $\Gamma_i^\sigma, \mathcal L^\sigma$. 
 We 
  define the stack $\mathcal M_{dR}(X,r, \Gamma_i)$ respectively
 $\mathcal M_{dR}(X,r, \Gamma_i ,\mathcal L)$ of locally free connections with log poles along $D$ and eigenvalues of the residues being the sets $\Gamma_i$, resp. additionally with determinant $\mathcal L$. 
 The   complex analytic  Riemann-Hilbert map $
 \mathcal M_{dR}(X,r, \Gamma_i)(\mathbb C)\to M_B(X,r, [R_i])(\mathbb C)$ respectively $ \mathcal M_{dR}(X,r, \Gamma_i, \mathcal L)(\mathbb C)\to M_B(X,r, [R_i], \delta)(\mathbb C)$
 sending 
  $ \mathcal A=(A,\nabla)$ to  $\mathbb L=A|_X^\nabla$ is by   \cite[Corollary~5.6]{Del70} a bijection on points.
 Likewise, the restriction  $\mathcal M_{dR}(X,r, \Gamma_i) \to \mathcal M_{dR}(X,r)$ to the stack of regular singular rank $r$ connections on $X$  is an isomorphism onto its image.

In particular,  a privileged  $\mathbb L$ with respect to $[R_i]$ respectively $([R_i], \delta)$ yields  the only point of $\mathcal M_{dR}(X,r, \Gamma_i)$ with tangent space of dimension
${\rm dim}_{\mathbb C} H^1(\bar Y, \mathcal E nd(\mathbb M))^G$, respectively 
of $\mathcal M_{dR}(X,r, \Gamma_i, \delta)$ with tangent space of dimension
${\rm dim}_{\mathbb C} H^1(\bar Y, \mathcal E nd^0(\mathbb M))^G$, see Fact~\ref{fact:dR}.

Finally,  for $\Gamma_i^{\mathcal D}$, the connection  $\mathcal A$ is 
  Deligne's extension $\mathcal E$ and in all cases $\mathcal A|_X=\mathcal E|_X$.

  \subsection{Good models} \label{ss:good_mod} We fix $[R_i]$ of finite order. 
 Let $S\to {\rm Spec}( \Z)$ be a smooth morphism, such that  $j$ has a good model 
 $X_S \hookrightarrow \bar X_S$ where $ D_S=\bar X_S\setminus X_S=
\cup_i D_{i,S}$  is a strict relative normal crossings divisor, as well as $Y_S\hookrightarrow \bar Y_S$ with relative normal crossings divisor $\bar q^{-1}(D_S)\cup R_S$ and such that the characteristic of all the closed  points of $S$  verify the conditions in Lemma~\ref{lem:arith=geom}. We assume $\mathbb L$ is privileged with respect  to $[R_i]$ respectively $([R_i],\delta)$. 
We assume in addition that
 we have  good models $\mathcal E_S$, $\mathcal E'_S
$, $\mathcal N_S$
$\mathcal N'_S$, of   their restriction to $X_{W(s)}, Y_{W(s)}$ and that we have base change for the de Rham cohomologies occurring in   Fact~\ref{fact:dR} on all $\mathcal \in M_{dR}(\bar Y, r)$ invariant by $G$. 

Let $s\in {\rm Spec}(S)$ be a closed point of characteristic $p$. 
  We require in addition that   $\mathcal E$  in
   $\mathcal M_{dR}( X, r,  \Gamma_i^{\mathcal D})$ respectively $\mathcal M_{dR}( X ,r,  \Gamma_i^{\mathcal D} ,\mathcal L^{\mathcal D})$  and   $\mathcal E'$ in $\mathcal M_{dR}( X, r,  \Gamma_i^{' \mathcal D})$ respectively $\mathcal M_{dR}( X ,r,  \Gamma_i^{ ' \mathcal D} ,\mathcal L^{' \mathcal D})$
  on $\bar X$, extend to   closed $S$-sections 
$\mathcal{E}_S , \mathcal{E}'_S$  in $ \mathcal M_{dR}( X,r, \Gamma^{\mathcal D}_i)_S$, $
 \mathcal M_{dR}( X,r, \Gamma^{' \mathcal D}_i)_S$, 
  respectively in $ \mathcal M_{dR}(X,r,  \Gamma^{\mathcal D}_i, \mathcal L^{\mathcal D})_{S},
  \mathcal M_{dR}(X,r,  \Gamma^{' \mathcal D}_i, \mathcal L^{' \mathcal D})_{S}
  $.
\\ \ \\
Recall from  \cite[Section~3.1]{EG25} that the Frobenius pullback $F^*$  of connections on vector bundles on $X_{W(s)}$ are defined uniquely by  the Frobenius pullback of lifts on formal neighbourhoods in  $\bar X_{W(s)}$ of the geometric Frobenius on $\bar X_s$  if the characteristic of $s$ is $>2$,  and that residues of log connections are  also recognized uniquely by their restriction to formal neighbourhoods of closed points.

\begin{thm} \label{thm:arith_frob}
Let  $\mathbb L$  be a privileged local system with respect to $[R_i]$ respectively $([R_i], \delta)$ with $[R_i]$ and $\delta$ of finite order.  Then if $\mathcal E|_{X_s}$ has vanishing $p$-curvature, then 
$ F^* \mathcal E_{W(s)}'|_{X_{W(s)}}=\mathcal E_{W(s)}|_{ X_{W(s)}}. $
\end{thm}
  \begin{proof}
  The vanishing of the $p$-curvature of $\mathcal E|_{X_s}$ implies that there is a locally free connection $\mathcal B$ on $X_{W(s)}$ such that $\mathcal E|_{X_{W(s)}}=F^* \mathcal B^\circ$.  We want to identify $\mathcal  B^\circ$ with $\mathcal E'_{W(s)}|_{X_{W(s)}}$.  Set $\mathcal B^\circ=(B^\circ, \nabla^\circ)$. We define  $B\subset j_{W(s) *} \mathcal B^\circ$ as the subsheaf of local sections $\tau$ such that $F^*\tau\in E_{W(s)}$. As $F$ is flat,  $E_{W(s)}$ is locally free and has log poles along $D_{W(s)}$,  the same holds for  $\mathcal B=(B, \nabla^\circ|_B)=(B, \nabla)$.  We want to identifiy $\mathcal B$. We denote by $\Gamma_i^{\mathcal B}$  the eigenvalues of its residues and by $\mathcal L^{\mathcal B}$ its determinant. 
  
  \medskip
  
We first assume that $X=\bar X$.
As  ${\rm dim}_{K(s)}H^1(X, \mathcal E nd(\mathcal  B))= {\rm dim}_{K(s)} F^* 
H^1(X, \mathcal E nd(\mathcal  B))$ where $K(s)={\rm Frac}(W(s))$, 
respectively ${\rm dim}_{K(s)}H^1(X, \mathcal E nd^0(\mathcal  B))= {\rm dim}_{K(s)} F^* 
H^1(X, \mathcal E nd^0(\mathcal  B))$,
we conclude that 
$F^* \mathcal B$ on $X_{K(s)}$ is equal to $\mathcal E$ on $X_{K(s)}$. As $\mathcal E|_{X_s}=F^*B|_{X_s}$ is stable, so is $B|_{X_s}$ thus  $\mathcal E'=\mathcal B$.
\\ \ \\
We now assume that $D\neq \varnothing$. 
In a formal neighbourhood $\mathcal U$ of a $W(s)$-point $a$ of $D_{i,W(s)} \setminus \cup_{j\neq i} D_{j,W(s)}$  defined by $x_i=0$, and  in a good basis,
   $ \mathcal B $   has  the  matrix $ \Psi_i d \log x_i$  where $\Psi_i$ is semi-simple with eigenvalues $\Gamma_i^{ \mathcal B}$ in  
  $ (\Z_p \cap \mathbb Q)^r$.
      The absolute Frobenius $F$ of $\bar X_s$ lifts to $F_{\mathcal U}$ on  $\mathcal U$ sending $x_i$ to $x_i^p$. Thus  
   \ga{}{  F_{\mathcal U}^* \Psi_i d\log x_i=p\Psi_i d\log x_i \ {\rm and} \ p\Gamma_i^{\mathcal B}=\Gamma_i^\mathcal D. \notag}
  By Lemma~\ref{lem:arith=geom} it implies that $\Gamma_i^{\mathcal B}$ is one of the possible $\Gamma'_i$ in Section~\ref{ss:dRst}.
     By the good reduction assumption for $H^1$ we have
$H^1_{dR}(\bar Y_{W(s)},   \mathcal N_{W(s)}^{\vee}\otimes \mathcal N_{ W(s)} )^G=
H^1_{dR}(\bar Y_{W(s)},   \mathcal N_{W(s)}^{'\vee}\otimes \mathcal N'_{ W(s)} )^G
 $
and the latter is equal to $F^* H^1_{dR}(\bar Y_{W(s)},   \mathcal N_{W(s)}^{'\vee}\otimes \mathcal N'_{ W(s)} )^G$.
     This implies that $\mathcal B^\circ$ has to be one of the crystals associated to the privileged local systems and this crystal has to be $\mathcal E'|_{X_{W(s)}}$. One argues similary in the trace free case with fixed determinant. 
\end{proof}

\begin{rmk} \label{rmk:int}
The vanishing of the $p$-curvature of $X_s$ is used to obtain Theorem~\ref{thm:arith_frob} integrally. 
Without this assumption, the same proof yields that for an  irreducible  local system $\mathbb L$, privileged with respect to  $[R_i]$ resp. $([R_i], \delta)$ with $[R_i]$  and $\delta$ of finite order, then  
$ F^* \mathcal E_{K(s)}'|_{X_{K(s)}}=\mathcal E_{W(s)}|_{ X_{K(s)}}. $
\end{rmk}

 \begin{cor} \label{cor:frob:N} Under the assumptions of Theorem~\ref{thm:arith_frob}
 it holds 
 \begin{itemize}
 \item[1)]
  $ F^* \mathcal N'_{W(s)}=\mathcal N_{W(s)}, \ 
  (\bar q_*F^* \mathcal N'_{W(s)})^G=\mathcal E_{W(s)}
     .$
   \item[2)]  $N'_s$ on $\bar Y_s$ is a sum of stable vector bundles with vanishing numerical Chern classes.
   \item[3)] $N'$ on $Y$ is a sum of stable vector bundles with vanishing numerical Chern classes.
   \end{itemize}
 \end{cor} 
  \begin{proof}
 By Theorem~\ref{thm:arith_frob}, $ F^* \mathcal N_{W(s)}'|_{ Y_{W(s)}}
 =\mathcal N_{W(s)}|_{ Y_{W(s)}}$. Consequently the Deligne extensions  are equal. 
  The connections on either side having no poles, this proves the first equality. As for the second one, it follows from the compatibility of the Deligne extension with push-forward. This proves 1). 
  
  The first equality also implies that the connection on $\mathcal N'_{s}$ is a direct  sum of stable  canonical connections with vanishing  numerical Chern classes, and its Cartier descent is precisely the vector bundle $N'_s$.  So  the Cartier descent $N'_s$ is semi-stable as a vector bundle with vanishing numerical Chern classes  as a vector bundle.  This proves 2). 
  
The vector bundle $N'_S$ on $\bar Y_S$ is semi-stable and  has vanishing numerical Chern classes on $\bar Y_s$  for all closed points $s\in S$, so it has the same property on $\bar Y_S$.  
This proves 3). 
  \end{proof}
 \section{Proof of the main  Theorem} \label{sec:proof}
\begin{proof}[Proof of Theorem~\ref{thm:main}] Let $\mathbb L$ be the given privileged local system with respect to $[R_i]$ respectively $([R_i], \delta)$, where $[R_i]$ and $\delta$ are of finite order,  such that the associated connection has vanishing $p$-curvature almost everywhere. We consider a model as in Section~\ref{sec:galois}. The orbit of $\mathbb L$ under ${\mathsf{ Gal}}(K/\mathbb Q)$ is finite. 
 We  choose $p$ as in Section~\ref{sec:galois}.
 By  
Theorem~\ref{thm:tpriv}
   all $\mathbb L^\sigma, \ \mathbb M^\sigma$ are polarized complex variations of Hodge structures.  
 By 
 Simspon's correspondence \cite[Section~4]{Sim92} the underlying connections are polystable  and the underlying vector bundles are polystable if and only if the Hodge fltration is trivial, that is  the local systems $\mathbb M^\sigma$ are unitary. 
  It follows from Corollary \ref{cor:frob:N}  that  all the $\mathbb L^\sigma$ are unitray local systems. 
On the other hand, by  Theorem~\ref{thm:int}, 
 $\oplus_{\sigma\in {\mathsf{ Gal}}(K/\mathbb Q)} \mathbb L^{\sigma}$  is strongly integral, i.e., defined over $\Zb$. We conclude using the well-known Lemma \ref{lemma:finite} below that
$\oplus_{\sigma\in {\mathsf{Gal}}(K/\mathbb Q)} \mathbb L^{\sigma}$ is finite. \end{proof}

\begin{lem}[Topological criterion for finiteness]\label{lemma:finite}
Let $\rho\colon \Gamma \to {\mathsf{GL}}_{r}(\Zb)$ be a representation of a  group $\Gamma$. Then, $\rho(\Gamma)$ is finite if and only if $\rho$ is conjugate to a unitary representation $\Gamma \to
 {\mathsf{ U}}(r)\subset {\mathsf{GL}}_r(\mathbb C)$.
\end{lem}
\begin{proof}
The assumption implies that the subgroup   $\rho(\Gamma)  \subset {\mathsf{GL}}_r(\Zb) \subset {\mathsf{GL}}_r(\mathbb C)$ is a discrete subset. Therefore, it is finite if and only if it is bounded.  Any bounded subgroup of ${\mathsf{GL}}_r(\mathbb C)$ is conjugate to a subgroup of  ${\mathsf{ U}}(r)$.
\end{proof}

\section{A $p$-adic alternative to conclude the finiteness of monodromy}

In this section, we briefly remark on a $p$-adic   alternative  to the topological  boundedness argument in Lemma \ref{lemma:finite}.  Although it is
 more involved than the brief topological argument, it fits in the spirit of this note. 
 The main observation
  states in a precise manner that for the overconvergent  $F$-isocrystals obtained by mod $p$ reduction of privileged flat connections, the formation of $\sigma$-companions matches with the process of $\sigma$-Galois conjugation of the Betti local systems, whenever $\sigma$ is an element of the Galois group of the field of definition of the Betti local system.

We take a  model $X_S$ etc. as in Section~\ref{sec:galois},  which is good now  for the finitely many  Deligne extensions $\mathcal E^\sigma$, the  $\mathcal N^\sigma$ on $\bar X$ and $\bar Y$,  all the summands of the latter ones and  has the property with  good sections for all $\mathcal E^\sigma$. 
 In the sequel,  $\mathcal E^\sigma_{W(s)}$  {\it is viewed as a crystal on } $X_s$, and $\mathcal E^\sigma_{K(s)}$ as the {\it  induced  overconvergent
$F$- isocrystal} on $X_s$.

\begin{prop}  \label{prop:comp}
The assumptions are as in Theorem~\ref{thm:main}. 
Let $\sigma \in \mathsf{Gal}(K/\mathbb{Q})$. Then, $\mathcal{E}_{K(s)}^{\sigma}$ is the $\tilde{\sigma}$-companion to $\mathcal{E}_{K(s)}$, where $\tilde{\sigma} \in \mathsf{Aut}(\bar{\mathbb{Q}}_p)$ is an arbitrary extension of $\sigma$ to an automorphism of $\bar{\mathbb{Q}}_p$.
\end{prop}
\begin{proof}

Using the same counting argument as in \cite[Section 7]{EG20}, one sees that the $\tilde{\sigma}$-companion of $\mathcal{E}$ exists. It follows from \cite[Theorem 9.8]{Del73} that the corresponding formal connection lies in $\mathcal M_{dR}(X,r, \Gamma_i^\sigma)$ resp.
$\mathcal M_{dR}(X,r, \Gamma_i^\sigma,\mathcal L^\sigma)$. It must be the privileged point, in case b) one uses that forming companions respects the intermediate cohomology, see \cite[Lemma~3.4]{EG18}.
\end{proof}

In the following we denote by $Z/\mathbb{F}_q$ a smooth scheme over a finite field. An overconvergent $F$-isocrystal  $\mathcal F$ on $Z$ is called \emph{bounded} (in the literature, this is referred to as the \emph{unit-root} property), if for every closed $z \in |Z|$ its Frobenius eigenvalues at $z$ are $p$-{\it adic units}.

\begin{prop}[Koshikawa]
Let $\mathcal{F}$ be an irreducible overconvergent $F$-isocrystal  with order determinant on $Z$ with the property that for every $\sigma \in \mathsf{Aut}(\bar{\mathbb{Q}}_p)$, the $\sigma$-companion $\mathcal{F}^{\sigma}$ of $\mathcal{F}$ is bounded. Then, $\mathcal{F}$ is isotrivial, i.e., there exists a finite \'etale cover $f\colon Z' \to Z$ such that $f^*\mathcal{F}$ is the trivial $F$-isocrystal.
\end{prop}

We refer the reader to \cite[Theorem 1.2]{Kos17} for a proof  and only briefly comment on the idea. The assumption of all $\sigma$-companions being bounded is equivalent to the Frobenius eigenvalues $\lambda$ at every closed point $z$ satisfying $|\sigma(\lambda)|=1,$ for all field automorphisms $\sigma$ of $\bar{\mathbb{Q}}_p$,  where $| - |$ denotes the $p$-adic norm.  As explained in \cite{Kos17}, this happens if and only if the Frobenius eigenvalues are roots of unity at all closed points. The isotriviality is then inferred from this property.

By Corollary \ref{cor:frob:N}, all companions of $\mathcal{E}_{K(s)}$ give rise to an underlying stable vector bundle, and thus are bounded, in the above sense. We may therefore apply Koshikawa's finiteness criterion.

\section{Application} \label{sec:appli}

Let $\bar f \colon \bar X\to S$ be a   smooth  smooth projective morphism  over a connected smooth quasi-projective variety $S$, such that $\bar f_*\sO_{\bar X}=\sO_S$, together with a relative normal crossings divisor $\cup_{i\in I} D_i=D \to S$. Define
$X=\bar X\setminus D, \ f=\bar f|_X \colon X\to S$. For $T\to S$ a morphism, we denote by $X_T=X\times_S T$ the pullback, and similarly $\bar X_T, \bar f_R, \bar f_T $ etc.   Let $s\in S$ be a closed point.  We fix a rank $r$ together with quasi-unipotent  conjugacy classes $[R]_i\subset {\mathsf{GL}}_r(\mathbb C), \ i\in I$ and a torsion character $\delta$.
 We assume that $S$ is a $K\pi_1$ or else that $X_s$ is an hyperbolic curve or else  that $f$ has a section. 
 \begin{thm} \label{thm:gk}

 Let $\mathbb L_s$ be an irreducible privileged local system  with respect to $[R_i]$, respectively $([R_i], \delta)$  on $X_s$  of rank $r$,  defined over the (then cyclotomic)  field $K$. Then there is a finite \'etale base change $T\to S$  such that 
 \begin{itemize}
 \item[1)] $\mathbb L_s$ extends to a local system $\mathbb L_T$ on $X_T$ defined over $K$;
 \item[2)] if $\mathbb L_T$ and $\mathbb L_{T'}$ are two such extensions, then there are $T''\to T', \ T''\to T$ finite \'etale such that the pullbacks of  $\mathbb L_T$ and $\mathbb L_{T'}$ to $T''$ are isomorphic.
 \item[3)] if $\mathbb L_s$  has $p$-curvature equal to $0$ for almost all $p$, so does $\mathbb L_T$ and both verify the $p$-curvature conjecture. 
 \end{itemize}
 \end{thm}
 \begin{proof} 
 We apply \cite[Theorem~1.4]{EK24}.  
 We argue as in Proposition~\ref{prop:sigma}
  with ${\mathsf{ Aut}}(\mathbb C)$ replaced by  
  $\pi_1(S,s)$ to conclude that its  action $\mathbb L_s$ is fixed by this ation. 
  This proves 1).  The essential uniqueness in i) of \cite[Theorem~1.4]{EK24} proves 2). 
Finally for 3), the $p$-curvature vanishing implies by Theorem~\ref{thm:main} that the monodromy of $\mathbb L_s$ is finite. Then we  apply \cite[Theorem~1.4 c)]{EK24} to conclude that the one of $\mathbb L_T$ is finite as well.  In particular $\mathbb L_T$ has vanishing $p$-curvature almost everywhere. 
 \end{proof}

 \subsection{A concrete example} \label{ss:Pat} 
One example is concretely written in the literature, see \cite[Theorem~24]{Pat14}. We reproduce here  Patrikis' example on the  Betti side  (rather than on the  $\ell$-adic one).
 Then  $ S={\rm Spec}( B)$ with 
 \ga{}{ B=\C[T_1,\ldots, T_n][1/\prod_{i<j} (T_i-T_j)],  X= {\rm Spec}(A), \ A=B[X][1/\prod_i (X-T_i)], \notag}
 and $\mathbb L_s$ is one of Katz's rigid (so privileged) local systems on $X_s$. So Theorem~\ref{thm:gk}  3) applies. 
 
 \section{Variant}
Let  $(X,D)$ be as  in Theorem~\ref{thm:main} and  $\mathbb V$ be a semi-simple complex  local system with the property that the corresponding connection $(V,\nabla)$ admits a  model  $(X_S,  \mathcal V=(V,\nabla)_S)$ 
 as in Section~\ref{ss:good_mod}. We freely use the notation of Section~\ref{ss:good_mod}. Recall that for ${\rm Spec}(W(s))\to S$ given, $\mathcal V_{|_{X_{W(s)}}}$ is defined to be an $F$-crystal if 
 $F^* \mathcal V_{|_{X_{W(s)}}}=\mathcal V_{|_{X_{W(s)}}}$. 
 
 \begin{prop} \label{prop:V}
 If the connections $\mathcal V|_{X_{W(s)}}$ on $X_{W(s)}$ defined for a $W(s)$-lift satisfy $F^*\mathcal V|_{X_{W(s)}} = \mathcal V|_{X_{W(s)}}$
 over a dense set of closed points $s\in S$, and $\mathbb V$ is defined over $\mathbb Z$, then $(V,\nabla)$ verifies the $p$-curvature conjecture. 
 \end{prop}
  \begin{proof}
 As in Section~\ref{sec:galois} we go up to $\bar Y$ to trivialize the monodromies at infinity using  Proposition~\ref{prop:finite_inf}, and the $F$-crystal property translates on $\bar Y$ as an $F$-crystal property on the projective variety.  This implies in particular that the determinants of the summands on $\bar Y$  have finite order. We then argue as in Corollary~\ref{cor:frob:N}
to conclude that for this dense set of $s$,   $E_s$ is a  stable vector bundle with vanishing numerical Chern classes, and as in the Proof of Theorem~\ref{thm:main} in Section~\ref{sec:proof} to conclude that  $(V,\nabla)$ is unitary. Finally we apply Lemma~\ref{lemma:finite} to conclude. 
 \end{proof}

\begin{thm} \label{thm:subquotient}  Let $\mathbb V$ be as in Proposition~\ref{prop:V}. 
Then any irreducible  subquotient $\mathbb L$  of $\mathbb V$ which has monodromies $[R_i]$ along $D_i$ such that no other subquotient has the same monodromies along all $D_i$   satisfies the $p$-curvature conjecture. 
\end{thm}
 
 \begin{proof}
 If $\mathbb V$ itself is  irreducible, then $\mathbb L=\mathbb V$ is the only subquotient and there is nothing to prove. 
Otherwise, $\mathbb L$ is a summand  and has finite determinant. Thus, the proof of Corollary~\ref{cor:cyclo} applies and $\mathbb L$ is defined over a cyclotomic field. And then the proofs of 
 Corollary~\ref{cor:frob:N} and Section~\ref{sec:proof} remain the same when we do not consider the whole Betti moduli $M_B(X,r, [R_i]
 )$ but just $\mathbb L$ and its Galois conjugates $\mathbb L^\sigma$. 
 \end{proof}
 
 As a corollary we answer positively a question of Daniel Litt which he raised after seeing Theorem~\ref{thm:main}.
 
 \begin{cor}
 In Theorem~\ref{thm:subquotient} we may take $\mathbb V$ to be the Gauss-Manin system of a smooth projective family over $X$. 
 \end{cor}
  \begin{proof}
  
  The Gauss-Manin connection on $X_{W(s)}$ defines an $F$-crystal for almost all $s$. 
  By \cite[Introduction]{Kat72}  or simply Proposition~\ref{prop:V},  it verifies the $p$-curvature conjecture. We apply Theorem~\ref{thm:subquotient} to conclude.

  \end{proof}


\begin{thebibliography}{BBDE04}

\bibitem[Chu85]{Chu85} Chudnovsky, D. V., Chudnovsky, G. V.  {\it  Applications of Padé approximations to the Grothendieck conjecture on linear differential equations},  Lecture Notes in Mathematics {\bf 1135}, 52–100. 
 
 \bibitem[Del70]{Del70} Deligne, P.: {\it 
Équations Diﬀérentielles à Points Singuliers Réguliers}, Lecture Notes in Mathematics {\bf 163} (1970), 136 pp.
\bibitem[Del71]{Del71} Deligne, P.: {\it Théorie de Hodge, II}, Publ. math. I.H.É.S. {\bf 40} (1971), 5--57.

\bibitem[Del73]{Del73} Deligne, P. {\it Les constantes des équations fonctionnelles des fonctions L}, Modular Functions of One Variable II: Proceedings International Summer School University of Antwerp, RUCA July 17–August 3, 1972. Berlin, Heidelberg : Springer Berlin Heidelberg, 1973. p. 501-597.



\bibitem[EG18]{EG18}  Esnault, H., Groechenig, M.: {\it Cohomologically rigid connections and integrality,} 
Selecta Mathematica {\bf 24} (5) (2018), 4279--4292.
\bibitem[EG20]{EG20} Esnault, H., Groechenig, M. {\it Rigid connections and  
F
 -isocrystals}, Acta Math. {\bf 225}(1), 103-158, (2020).

\bibitem[EG21]{EG21} Esnault, H, Groechenig, M.: {\it Frobenius structures and unipotent monodromy at infinity}, preprint 2021, 8 pages, {\url{https://page.mi.fu-berlin.de/esnault/preprints/helene/143_appPST.pdf}}, 
Appendix to 'André-Oort for Shimura varieties' by J. Pila, Ananth Shankar, J. Tsimerman, 25 pages. 

\bibitem[EG25]{EG25} Esnault, H., Groechenig, M.: {\it Crystallinity of rigid flat connections revisited,}
Selecta  {\bf 31}  (2025), No 1, Paper 2, 38 pp.

\bibitem[EK24]{EK24}  Esnault, H., Kerz, M.: {\it A non-abelian version of Deligne's Fixed Part Theorem}, preprint 2024, 23 pages
appears in Algebraic Geometry, {\url{https://page.mi.fu-berlin.de/esnault/preprints/helene/152_non_ab_Hodge.pdf}}





\bibitem[Fae24]{Fae24} Færgeman, J.: {\it Motivic realization of rigid $G$-local systems on curves and tamely ramified geometric Langlands}, {\url{https://arxiv.org/pdf/2405.18268}}.

\bibitem[FK09]{FK09} Farb  B., Kisin, M.:  {\it The Grothendieck-Katz conjecture of certain locally symmetric varieties},   IMRN {\bf 22} (2009),  4159--4167.
\bibitem[Kat72]{Kat72} Katz, N. {\it Algebraic Solutions of Diﬀerential Equations ($p$-curvature and the Hodge
Filtration)}, Invent. math. {\bf 18} (1972), 11-118.

\bibitem[Kat82]{Kat82}  Katz, N.: {\it A conjecture in the arithmetic of diﬀerential equations}, Bull. S.M.F.  {\bf 110}
(1982), 203--239.
\bibitem[Kat96]{Kat96} Katz, N.: {\it Rigid local systems}, Ann. of Math. Stud. {\bf 139}, Princeton University Proess, NJ, 1996, viii + 223 pp.

\bibitem[Kaw81]{Kaw81} Kawamata, Y.: {\it Characterization of abelian varieties}, Compos. Math. {\bf 43} (1981), 253--276.

\bibitem[Kos17]{Kos17} Koshikawa, T.: {\it Overconvergent unit-root $F$-isocrystals and isotriviality}, Mathematical Research Letters, 2017, vol. {\bf 24}, no 6, p. 1707-1727.

\bibitem[Mat02]{Mat02} Matsuki, K.: {\it Introduction to the Mori program}, Universitext, Springer Verlag 2002.


\bibitem[Pat14]{Pat14} Patrikis, S.:  {\it Pure Motives and Rigid Local Systems}, Notes Harvard 2014,  {\url{https://www.math.columbia.edu/~chaoli/docs/PureMotives.html#sec23}}.


 \bibitem[Sim92]{Sim92} Simpson, C.: {\it  Higgs bundles and local systems}, Publ. math. I.H.\'E.S. {\bf 75} (1992), 5--95.
\end{thebibliography}
\end{document}